\documentclass[reqno,12pt]{amsart}

\usepackage{amsmath, epsfig, cite}
\usepackage{amssymb}
\usepackage{amsfonts}
\usepackage{latexsym}
\usepackage{amsthm}

\newtheorem{theorem}{Theorem}[section]

\newtheorem{corollary}[theorem]{Corollary}

\newtheorem{lemma}[theorem]{Lemma}

\theoremstyle{remark}

\numberwithin{equation}{section}

\newcommand{\Z}{\mathbb Z}

\allowdisplaybreaks

\author{Xiaoxia Wang }
\address{DEPARTMENT OF MATHEMATICS, SHANGHAI UNIVERSITY, SHANGHAI 200444, P. R. CHINA}
\email{xiaoxiawang@shu.edu.cn (X. Wang), xchangi@shu.edu.cn (C. Xu).}

\author{Chang Xu$^*$}

\title{Some $q$-supercongruences on double and triple sums}



\begin{document}
\begin{abstract}
In this paper, we investigate a number of $q$-supercongruences on double and triple sums. 
By means of a lemma devised by  El Bachraoui and its generalization, we transform some
$q$-supercongruences on double and triple sums into 
the $q$-supercongruences of the square and cube of truncated basic hypergeometric series.
Our main tools still involve
 the `creative microscoping' method introduced by Guo and Zudilin,
a lemma designed by Guo and Li and the Chinese remainder theorem.
\end{abstract}
\maketitle
\section{Introduction}
It was a family of formulas for $1/\pi$ that were discovered  by Ramanujan \cite{Ramanujan}.
Based on these formulas, several Ramanujan-type supercongruences were presented in \cite{Van},
such as,
\begin{align*}\label{EE5}
&\mathrm{(B.2)}\quad\sum_{k=0}^{\frac{p-1}{2}}(-1)^k(4k+1)\frac{(\frac{1}{2})^3_k}{k!^3}
\equiv\frac{-p}{\Gamma_p(\frac{1}{2})^2}\pmod{p^3},\\
&\mathrm{(F.2)}\quad\sum_{k=0}^{\frac{p-1}{4}}(-1)^k(8k+1)\frac{(\frac{1}{4})^3_k}{k!^3}
\equiv\frac{-p}{\Gamma_p(\frac{1}{4})\Gamma_p(\frac{3}{4})}\pmod{p^3},
\quad\mathrm{if}\quad p\equiv 1\pmod4,\\
&\mathrm{(G.2)}\quad\sum_{k=0}^{\frac{p-1}{4}}(8k+1)\frac{(\frac{1}{4})^4_k}{k!^4}
\equiv\frac{p\Gamma_p(\frac{1}{4})\Gamma_p(\frac{1}{2})}{\Gamma_p(\frac{3}{4})}\pmod{p^3},
\quad\mathrm{if}\quad p\equiv 1\pmod4.
\end{align*}
Here and in what follows, $p$ is an odd prime, $(a)_n=a(a+1)\cdots(a+n-1)$ is the Pochhammer symbol, and $\Gamma_p(\cdot)$ is
 the $p$-adic gamma function.

In the past decades,  some progress on supercongruences and their $q$-analogues has been made, the reader may find some of the results in \cite{Guo2,Wang_Yue,Wang_Yu,Xu_Wang}.
Particularly, Guo \cite[Theorem 1.2]{Guo5} gave a $q$-analogue of (B.2) as follows: for any odd integer  $n>1$,
\begin{equation}\label{EEE5}
\sum_{k=0}^{\frac{n-1}{2}}(-1)^kq^{k^2}[4k+1]\frac{(q;q^2)^3_k}{(q^2;q^2)^3_k}
\equiv
(-1)^{\frac{n-1}{2}}q^{\frac{(n-1)^2}{4}}[n]
\pmod{[n]\Phi_n(q)^2}.
\end{equation}
Here and throughout the paper, $[n]=[n]_q=1+q+\cdots+q^{n-1}$ denotes the {\em $q$-integer},
and the {\em $q$-shifted factorial} is defined as
$(a;q)_0=1$ and $(a;q)_n=(1-a)(1-aq)\cdots(1-aq^{n-1})$ when  $n\in \mathbb{N}$.
Briefly, we write $(a_1,a_2,\ldots,a_m;q)_n=(a_1;q)_n (a_2;q)_n\cdots (a_m;q)_n$.
Moreover,  $\Phi_n(q)$ denotes the $n$-th {\em cyclotomic polynomial} in $q$,
which satisfies
$\Phi_n(q^2)=\Phi_n(q)\Phi_n(-q)$ for any odd integer $n$.

Applying the `creative microscoping' method introduced in \cite{Guo_Zudilin},
El Bachraoui \cite{Bachraoui} observed several supercongruences from the square of truncated
basic hypergeometric series.
Later,  Li \cite{LiLong} gave some new $q$-supercongruences. For example,
for any positive odd integer $n$,
\begin{equation} \label{EEE4}
\sum_{k=0}^{n-1}\sum_{j=0}^{k}
c_q(j)c_q(k-j)
\equiv
q^{\frac{n+1}{2}}[n]^2\pmod{[n]\Phi_n(q)^2},
\end{equation}
where $c_q(k)=(-1)^kq^{k^2}[4k+1](q;q^2)^3_k/(q^2;q^2)^3_k$.
Letting $n=p$, $q\rightarrow1$ in \eqref{EEE4}, Li \cite{LiLong} also obtained the following supercongruence:
\begin{flalign}
\sum_{k=0}^{p-1}\frac{(-1)^k}{64^k}
\sum_{j=0}^{k}
{\dbinom{2j}{j}}^3
{\dbinom{2k-2j}{k-j}}^3(4j+1)(4k-4j+1)\equiv
p^2\pmod{p^3}.\label{eq2}
\end{flalign}

In this paper,
by some essential tools including
the `creative microscoping' method, we shall
investigate several $q$-supercongruences on double sums similar to \eqref{EEE4}.

Our first two results can be stated as follows.
\begin{theorem}\label{Thm3}
For odd integers $n>1$ and integers $k\geq0$,
\begin{equation*}
c_q(k)=(-1)^k\frac{(1+q^{4k+1})(q^{2};q^{4})^2_k}{(1+q)(q^{4};q^{4})^2_k}q^{2k^2+k},
\end{equation*}
there holds, modulo $\Phi_n(-q)^3\Phi_n(q)^2$,
\begin{flalign}
\sum_{k=0}^{n-1}\sum_{j=0}^{k}c_q(j)c_q(k-j)
\equiv
[n]^2_{q^2}q^{(n-1)^2}
\bigg(\sum_{k=0}^{\frac{n-1}{2}}\frac{(q^{2};q^{4})_k}
{[4k+1](q^{4};q^{4})_k}q^{2k}\bigg)^2.\label{EE2}
\end{flalign}
\end{theorem}

\begin{theorem}\label{Thm4}
For odd integers $n>1$ and integers $k\geq0$,
\begin{equation*}
c_q(k)=(-1)^k\frac{(1+q^{4k+1})(q^{2};q^{4})^3_k}{(1+q)(q^{4};q^{4})^3_k}q^{2k^2+2k},
\end{equation*}
there holds, modulo $\Phi_n(-q)^4\Phi_n(q)^2$,
\begin{flalign}
\sum_{k=0}^{n-1}\sum_{j=0}^{k}c_q(j)c_q(k-j)
\equiv
[n]^2_{q^2}q^{(n-1)^2}
\bigg(\sum_{k=0}^{\frac{n-1}{2}}\frac{(q^{2};q^{4})^3_k}
{(q^{4},q^{3},q^{5};q^{4})_k}q^{2k}\bigg)^2.\label{E7}
\end{flalign}
\end{theorem}

In fact,  Theorems \ref{Thm3} and \ref{Thm4} can be generalized into the following two results,
 among which the $d=2$ and $r=1$ case is a weaker version of Theorems \ref{Thm3} and \ref{Thm4}.
\begin{theorem}\label{Thm1}
Let $n>1$  be an odd integer and $d$ a  positive integer with $\gcd(n,d)=1$.
Let $r$ be an integer  with $n-(n-1)d/2\leq r\leq n$ and $n\equiv r\pmod d$.
For $k\geq0$,
\begin{equation*}
c_q(k)=(-1)^k\frac{(1+q^{2dk+r})(q^{2r};q^{2d})^2_k}{(1+q^r)(q^{2d};q^{2d})^2_k}q^{dk^2+k(d-r)},
\end{equation*}
there holds, modulo $\Phi_n(-q)^2\Phi_n(q)^2$,
\begin{flalign}
\sum_{k=0}^{n-1}\sum_{j=0}^{k}c_q(j)c_q(k-j)
\equiv
\frac{[n]^2_{q^2}}{[r]^2_{q^2}}q^{\frac{2(n-r)(n+r-d)}{d}}
\bigg(\sum_{k=0}^{\frac{n-r}{d}}\frac{(q^{2r},q^r;q^{2d})_k}{(q^{2d},q^{2d+r};q^{2d})_k}q^{2k(d-r)}\bigg)^2.\label{EE1}
\end{flalign}
\end{theorem}

\begin{theorem}\label{Thm2}
Let $n>1$  be an odd integer and $d$  a positive integer with $\gcd(n,d)=1$.
Let $r$ be an integer  with $n-(n-1)d/2\leq r\leq n$ and $n\equiv r\pmod d$.
For $k\geq0$,
\begin{equation*}
c_q(k)=(-1)^k\frac{(1+q^{2dk+r})(q^{2r};q^{2d})^3_k}{(1+q^r)(q^{2d};q^{2d})^3_k}q^{dk^2+2k(d-r)},
\end{equation*}
there holds, modulo $\Phi_n(-q)^3\Phi_n(q)^2$,
\begin{flalign}
\sum_{k=0}^{n-1}\sum_{j=0}^{k}c_q(j)c_q(k-j)
\equiv
\frac{[n]^2_{q^2}}{[r]^2_{q^2}}q^{\frac{2(n-r)(n+r-d)}{d}}
\bigg(\sum_{k=0}^{\frac{n-r}{d}}\frac{(q^{2r};q^{2d})^2_k(q^d;q^{2d})_k}
{(q^{2d},q^{d+r},q^{2d+r};q^{2d})_k}q^{2k(d-r)}\bigg)^2.\label{E2}
\end{flalign}
\end{theorem}
Note that, recently, taking the advantage of  the Bailey transformation formula  \cite[Eq. (2.9)]{Andrews},
the authors \cite{Wang_Xu} generalized Guo's results \cite[Theorems 1.1 and 1.2]{Guo6} as follows:
\begin{flalign}
&\sum_{k=0}^{\frac{n-r}{d}}(-1)^k\frac{(1+q^{2dk+r})(q^{2r};q^{2d})^2_k}
{(1+q^r)(q^{2d};q^{2d})^2_k}q^{dk^2+k(d-r)}\nonumber\\
&\equiv(-1)^{\frac{n-r}{d}}\frac{[n]_{q^2}}{[r]_{q^2}}q^{\frac{(n-r)(n+r-d)}{d}}
\sum_{k=0}^{\frac{n-r}{d}}\frac{(q^{2r},q^r;q^{2d})_k q^{2k(d-r)}}
{(q^{2d},q^{2d+r};q^{2d})_k} \pmod{\Phi_n(-q)^3\Phi_n(q)^2},\label{EEE8}\\
&\sum_{k=0}^{\frac{n-r}{d}}(-1)^k\frac{(1+q^{2dk+r})(q^{2r};q^{2d})^3_k}{(1+q^r)(q^{2d};q^{2d})^3_k}q^{dk^2+2k(d-r)}\nonumber\\
&\equiv(-1)^{\frac{n-r}{d}}\frac{[n]_{q^2}}{[r]_{q^2}}q^{\frac{(n-r)(n+r-d)}{d}}
\sum_{k=0}^{\frac{n-r}{d}}\frac{(q^{2r};q^{2d})^2_k(q^d;q^{2d})_k q^{2k(d-r)}}
{(q^{2d},q^{d+r},q^{2d+r};q^{2d})_k}\!\!\pmod{\Phi_n(-q)^3\Phi_n(q)^2},\label{EEE7}
\end{flalign}
where $n$ is a positive odd integer, $d>1$ is an integer with $\gcd(n,d)=1$, $r$ is an integer with $r<n$ and $n\equiv r\pmod d$.

Clearly, the $d=2$ and $r=1$ case of \eqref{EEE7} is a new $q$-analogue of van Hamme's  (B.2).
And letting $d=4$ and $r=1$ in \eqref{EEE7} produces a $q$-analogue of (F.2).

Moreover, making suitable substitutions in Theorems \ref{Thm3}--\ref{Thm2}, we can derive some new supercongruences.
For instance,
replacing $n$ by odd prime $p$ and setting $q\rightarrow 1$ in Theorems \ref{Thm3} and \ref{Thm4} respectively,
 we get the following supercongruences.
\begin{corollary}
For any odd prime $p$, we have
\begin{flalign*}
&\sum_{k=0}^{p-1}\frac{(-1)^k}{16^k}
\sum_{j=0}^{k}
{\dbinom{2j}{j}}^2
{\dbinom{2k-2j}{k-j}}^2
\equiv
p^2\bigg(\sum_{k=0}^{\frac{p-1}{2}}\frac{1}{(4k+1)4^k}{\dbinom{2k}{k}}\bigg)^2\pmod{p^2},\\
&\sum_{k=0}^{p-1}\frac{(-1)^k}{64^k}
\sum_{j=0}^{k}
{\dbinom{2j}{j}}^3
{\dbinom{2k-2j}{k-j}}^3
\equiv
p^2\bigg(\sum_{k=0}^{\frac{p-1}{2}}\frac{(\frac{1}{2})^2_k}{(\frac{3}{4})_k(\frac{5}{4})_k4^k}\dbinom{2k}{k}\bigg)^2\pmod{p^2}.
\end{flalign*}
\end{corollary}
Likewise, replacing $n$ by odd prime $p$ and setting $q\rightarrow -1$ in Theorems \ref{Thm3} and \ref{Thm4} respectively,
 we obtain the following supercongruences.
 \begin{corollary}
 For any odd prime $p$, we have
\begin{flalign}
&\sum_{k=0}^{p-1}\frac{1}{16^k}
\sum_{j=0}^{k}
{\dbinom{2j}{j}}^2
{\dbinom{2k-2j}{k-j}}^2(4j+1)(4k-4j+1)
\equiv0\pmod{p^3},\label{E12}\\
&\sum_{k=0}^{p-1}\frac{(-1)^k}{64^k}
\sum_{j=0}^{k}
{\dbinom{2j}{j}}^3
{\dbinom{2k-2j}{k-j}}^3(4j+1)(4k-4j+1)\equiv
p^2\pmod{p^4}.\label{E13}
\end{flalign}
\end{corollary}
It is worth noting that \eqref{E12} follows from the result
\begin{align*}
\sum_{k=0}^{\frac{p-1}{2}}\frac{1}{4^k}{\dbinom{2k}{k}}\equiv0\pmod p,
\end{align*}
which is a special case of the following congruence
confirmed by Sun
\cite[Theorem 1.1]{Sun}:
\begin{align*}
\sum_{k=0}^{p^a-1}\frac{\binom{2k}{k}}{m^k}
\equiv
\left(\frac{m^2-4m}{p^a}\right)
+
\left(\frac{m^2-4m}{p^{a-1}}\right)
u_{p-\left(\frac{m^2-4m}{p}\right)}(m-2,1) \pmod{p^2},
\end{align*}
where $p$ is an odd prime, $a\in \Z^{+}$, $m$ is any integer not divisible by $p$, $u_n=u_n(A,B)$ is the Lucas sequence.
More interestingly,  
it is our congruence \eqref{E13} that  is a stronger version of Li's congruence \eqref{eq2}.

Furthermore, setting $n=p$, $q\rightarrow-1$, $d=4$ and $r=1$ in Theorems \ref{Thm1} and \ref{Thm2},
respectively, provides the following results.
\begin{corollary}
For any prime $p\equiv1\pmod4$, we have,
\begin{flalign*}
&\sum_{k=0}^{p-1}
\sum_{j=0}^{k}
(8j+1)(8k-8j+1)\frac{(\frac{1}{4})^2_j(\frac{1}{4})^2_{k-j}}{j!^2(k-j)!^2}
\equiv
p^2\bigg(\sum_{k=0}^{\frac{p-1}{4}}\frac{(\frac{1}{4})_k}{k!}\bigg)^2\pmod{p^2},\\
&\sum_{k=0}^{p-1}(-1)^k
\sum_{j=0}^{k}
(8j+1)(8k-8j+1)\frac{(\frac{1}{4})^3_j(\frac{1}{4})^3_{k-j}}{j!^3(k-j)!^3}
\equiv
p^2 \pmod{p^3}.
\end{flalign*}
\end{corollary}


The rest of this paper is organized as follows. In Sections \ref{Section3} and \ref{Section4},
the detailed derivation of  Theorems \ref{Thm1} and \ref{Thm2} will be presented, respectively.
 The proofs of Theorems \ref{Thm3} and \ref{Thm4} will be shown in Section \ref{Section5}
building on the results from the previous section.
In the last section,  our strategies of proving the supercongruences will be developed and
new $q$-supercongruences on triple sums shall be given.

\section{Proof of Theorem \ref{Thm1}}\label{Section3}
In \cite{Bachraoui}, El Bachraoui introduced an elegant lemma.
Now, we are going to show a generalization of  this lemma, which plays a crucial role  in the derivation of our theorems.
\begin{lemma}\label{le1}
Let $n>1$  be an integer, $d\geq2$ an integer and $r$ an integer with $n-\lfloor(n-1)d/2\rfloor\leq r\leq n$ and $n\equiv r\pmod d$.
Let $\{c(k)\}^\infty_{k=0}$ be a sequence of complex numbers.
If $c(k)=0$ for $(n-r)/d< k\leq n-1$, then
\begin{align}
\sum_{k=0}^{n-1}\sum_{j=0}^{k}c(j)c(k-j)=\bigg(\sum_{j=0}^{\frac{n-r}{d}}c(j)\bigg)^2.\label{eq4}
\end{align}
Furthermore, if $c(ln+k)/c(ln)=c(k)$ for all nonnegative integers $k$ and $l$ such that $0\leq k\leq n-1$, then
\begin{align}
\sum_{j=0}^{ln+k}c(j)c(ln+k-j)=\sum_{t=0}^{l}c(tn)c((l-t)n)\sum_{j=0}^{k}c(j)c(k-j).\label{eq5}
\end{align}
\end{lemma}
\begin{proof}
 \eqref{eq4} is immediate from the assumption.
Now we are going to confirm \eqref{eq5} is true.
\begin{flalign*}
&\sum_{j=0}^{ln+k}c(j)c(ln+k-j)\\
&=\sum_{j=0}^{k}c(j)c(ln+k-j)+\sum_{t=0}^{l-1}\sum_{j=tn+k+1}^{(t+1)n+k}c(j)c(ln+k-j)\\
&=c(ln)\sum_{j=0}^{k}c(j)c(k-j)+\sum_{t=0}^{l-1}\sum_{j=tn+k+1}^{(t+1)n-1}c(j)c(ln+k-j)+\sum_{t=0}^{l-1}\sum_{j=(t+1)n}^{(t+1)n+k}c(j)c(ln+k-j)\\
&=\sum_{t=0}^{l}c(tn)c((l-t)n)\sum_{r_j=0}^{k}c(r_j)c(k-r_j)+\sum_{t=0}^{l-1}c(tn)c((l-t-1)n)\sum_{r_j=k+1}^{n-1}c(r_j)c(n+k-r_j),
\end{flalign*}
where  $r_j$ is the remainder of the division of $j$ by $n$.

We now claim
\begin{equation}\label{EEEEEE1}
\sum_{t=0}^{l-1}c(tn)c((l-t-1)n)\sum_{r_j=k+1}^{n-1}c(r_j)c(n+k-r_j)=0.
\end{equation}
In fact, when $r_j>k\geq (n-r)/d$ or $r_j\geq (n-r)/d>k$, there holds
\begin{align*}
c(r_j)=0.
\end{align*}
When $(n-r)/d>r_j>k$, there holds
\begin{align*}
c(n+k-r_j)=0.
\end{align*}
As a result, \eqref{EEEEEE1} is correct. Then, we immediately obtain \eqref{eq5}.
\end{proof}

In order to prove Theorem \ref{Thm1}, we start by recalling a result from \cite[Eq. (3.5)]{Wang_Xu}.
\begin{lemma}[Wang and Xu\cite{Wang_Xu}]\label{Lemma2}
Let $n>1$ be an odd integer, $d$ a positive integer, $a$ an indeterminate, $r$ an integer with $r<n$ and $n\equiv r\pmod d$.
Then, modulo $\Phi_n(-q)(1-aq^{2n})(a-q^{2n})$,
\begin{flalign}
&\sum_{k=0}^{\frac{n-r}{d}}(-1)^k\frac{(1+q^{2dk+r})(aq^{2r},q^{2r}/a;q^{2d})_k}
{(1+q^r)(aq^{2d},q^{2d}/a;q^{2d})_k}q^{dk^2+k(d-r)}\nonumber\\
&\equiv(-1)^{\frac{n-r}{d}}\frac{[n]_{q^2}}{[r]_{q^2}}q^{\frac{(n-r)(n+r-d)}{d}}
\sum_{k=0}^{\frac{n-r}{d}}\frac{(aq^{2r},q^{2r}/a,q^r;q^{2d})_k}
{(q^{2d},q^{2dk+r},q^{2r};q^{2d})_k}q^{2k(d-r)}.\label{E1}
\end{flalign}
\end{lemma}
\begin{proof}[Proof of Theorem \ref{Thm1}]
Define the $k$-th term of left-hand side of congruence \eqref{E1} by $c_q(k,a)$, i.e.,
\begin{equation*}
c_q(k,a)=(-1)^k\frac{(1+q^{2dk+r})(aq^{2r},q^{2r}/a;q^{2d})_k}
{(1+q^r)(aq^{2d},q^{2d}/a;q^{2d})_k}q^{dk^2+k(d-r)}.
\end{equation*}
Then, putting $a=q^{2n}$ in \eqref{E1} gives
\begin{equation*}
\sum_{k=0}^{\frac{n-r}{d}}c_q(k,q^{2n})=(-1)^{\frac{n-r}{d}}\frac{[n]_{q^2}}{[r]_{q^2}}q^{\frac{(n-r)(n+r-d)}{d}}
\sum_{k=0}^{\frac{n-r}{d}}\frac{(q^{2r+2n},q^{2r-2n},q^r;q^{2d})_k}
{(q^{2d},q^{2r},q^{2d+r};q^{2d})_k}q^{2k(d-r)}.
\end{equation*}
Note the fact that $c_q(k,q^{2n})=0$ for $k$ in the range $(n-r)/d<k\leq n-1$.
Then, we obtain the following result via Lemma \ref{le1},
\begin{flalign*}
&\sum_{k=0}^{n-1}\sum_{j=0}^{k}c_q(j,q^{2n})c_q(k-j,q^{2n})\\
&=
\frac{[n]^2_{q^2}}{[r]^2_{q^2}}q^{\frac{2(n-r)(n+r-d)}{d}}
\bigg(\sum_{k=0}^{\frac{n-r}{d}}\frac{(q^{2r+2n},q^{2r-2n};q^{2d})_k(q^r;q^{2d})_k}{(q^{2d},q^{2r},q^{2d+r};q^{2d})_k}q^{2k(d-r)}\bigg)^2.
\end{flalign*}
Similarly, we can prove
 \begin{flalign*}
&\sum_{k=0}^{n-1}\sum_{j=0}^{k}c_q(j,q^{-2n})z_q(k-j,q^{-2n})\\
&=
\frac{[n]^2_{q^2}}{[r]^2_{q^2}}q^{\frac{2(n-r)(n+r-d)}{d}}
\bigg(\sum_{k=0}^{\frac{n-r}{d}}\frac{(q^{2r+2n},q^{2r-2n};q^{2d})_k(q^r;q^{2d})_k}{(q^{2d},q^{2r},q^{2d+r};q^{2d})_k}q^{2k(d-r)}\bigg)^2.
\end{flalign*}
Thus,
 we have the following $q$-congruence for the variable $a$: modulo $(1-aq^{2n})(a-q^{2n})$,
\begin{flalign}
\sum_{k=0}^{n-1}\sum_{j=0}^{k}c_q(j,a)c_q(k-j,a)
\equiv
\frac{[n]^2_{q^2}}{[r]^2_{q^2}}q^{\frac{2(n-r)(n+r-d)}{d}}
\bigg(\sum_{k=0}^{\frac{n-r}{d}}\frac{(aq^{2r},q^{2r}/a;q^{2d})_k(q^r;q^{2d})_k}
{(q^{2d},q^{2r},q^{2d+r};q^{2d})_k}q^{2k(d-r)}\bigg)^2.\label{eq10}
\end{flalign}
The proof then follows from the congruence \eqref{eq10} with $a\rightarrow 1$.
\end{proof}

\section{Proof of Theorem \ref{Thm2}}\label{Section4}
Before verifying Theorem \ref{Thm2}, we first list a result proved in\cite[Lemma 2]{Wang_Xu}.
\begin{lemma}[Wang and Xu\cite{Wang_Xu}]\label{Lemma3}
Let $n>1$ be an odd integer, $d$ a positive integer, $a$ an indeterminate, $r$ an integer with $r<n$ and $n\equiv r\pmod d$.
Then, modulo $\Phi_n(-q)(1-aq^{2n})(a-q^{2n})$,
\begin{flalign}
&\sum_{k=0}^{\frac{n-r}{d}}(-1)^k\frac{(1+q^{2dk+r})(aq^{2r},q^{2r}/a,q^{2r};q^{2d})_k}
{(1+q^r)(aq^{2d},q^{2d}/a,q^{2d};q^{2d})_k}q^{dk^2+2k(d-r)}\nonumber\\
&\equiv(-1)^{\frac{n-r}{d}}\frac{[n]_{q^2}}{[r]_{q^2}}q^{\frac{(n-r)(n+r-d)}{d}}
\sum_{k=0}^{\frac{n-r}{d}}\frac{(aq^{2r},q^{2r}/a,q^d;q^{2d})_k}
{(q^{2d},q^{d+r},q^{2d+r};q^{2d})_k}q^{2k(d-r)}.\label{E4}
\end{flalign}
\end{lemma}
\begin{proof}[Proof of Theorem \ref{Thm2}]
Let $c_q(k,a)$ denote the $k$-th term of left-hand side of congruence \eqref{E4}, i.e.,
\begin{equation*}
c_q(k,a)=(-1)^k\frac{(1+q^{2dk+r})(aq^{2r},q^{2r}/a;q^{2d})_k(q^{2r};q^{2d})_k}
{(1+q^r)(aq^{2d},q^{2d}/a;q^{2d})_k(q^{2d};q^{2d})_k}q^{dk^2+2k(d-r)}.
\end{equation*}
Since $n\equiv r\pmod{d}$, we have $(q^{2r};q^{2d})_k\equiv0\pmod{\Phi_n(q^2)}$ for $(n-r)/d<k\leq n-1$.
Then, according to Lemma \ref{le1}, we derive
\begin{flalign}
\sum_{k=0}^{n-1}\sum_{j=0}^{k}c_q(j,a)c_q(k-j,a)
\equiv
\bigg(\sum_{j=0}^{\frac{n-r}{d}}c_q(j,a)\bigg)^2
\equiv 0\pmod{\Phi_n(-q)}.\label{E5}
\end{flalign}
Note that $\sum_{j=0}^{\frac{n-r}{d}}c_q(j,a)\equiv0\pmod{\Phi_n(-q)}$
 has been proved in \cite{Wang_Xu}.

Next, along the same lines of the proof of Theorem \ref{Thm1}, we can obtain
the following $q$-congruence for the variable $a$: modulo $(1-aq^{2n})(a-q^{2n})$,
\begin{flalign}
\sum_{k=0}^{n-1}\sum_{j=0}^{k}c_q(j,a)c_q(k-j,a)
\equiv
\frac{[n]^2_{q^2}}{[r]^2_{q^2}}q^{\frac{2(n-r)(n+r-d)}{d}}
\bigg(\sum_{k=0}^{\frac{n-r}{d}}\frac{(aq^{2r},q^{2r}/a,q^d;q^{2d})_k}
{(q^{2d},q^{d+r},q^{2d+r};q^{2d})_k}q^{2k(d-r)}\bigg)^2.\label{E6}
\end{flalign}
It is easy to see that the right-hand side of \eqref{E6} is also congruent to $0$ modulo $\Phi_n(-q)$.
As a result, the congruence \eqref{E6} holds  modulo $\Phi_n(-q)$.
Since $1-aq^{2n}$, $a-q^{2n}$ and $\Phi_n(-q)$ are pairwise relatively prime polynomials,
we conclude \eqref{E6} is true modulo $\Phi_n(-q)(1-aq^{2n})(a-q^{2n})$.
Then, letting $a\rightarrow 1$ is led to the desired result.
\end{proof}

\section{ Proofs of Theorems \ref{Thm3} and \ref{Thm4}}\label{Section5}
The proofs of  Theorems \ref{Thm3} and \ref{Thm4} need following two auxiliary results.
Lemma \ref{Lemma5} is  given by Guo and Schlosser \cite[Lemma 2.1]{Guo_Schlosser1}
and has been widely used.
Lemma \ref{Lemma4} is recently established by Guo and Li \cite[Lemma 3.3]{Guo_Li}.
\begin{lemma}[Guo and Schlosser\cite{Guo_Schlosser1}]\label{Lemma5}
Let $n$, $d$ be positive integers, $r$ an integer satisfying $r<n$ and $n\equiv r\pmod d$. Then, for $0\le k\le (n-r)/d$,
we have
\begin{equation}\label{EE7}
\frac{(aq^r;q^d)_{(n-r)/d-k}}{(q^d/a;q^d)_{(n-r)/d-k}}
\equiv (-a)^{\frac{n-r}{d}-2k}\frac{(aq^r;q^d)_k}{(q^d/a;q^d)_k}q^{\frac{(n-r)(n-d+r)}{2d}+k(d-r)}\pmod{\Phi_n(q)}.
\end{equation}
\end{lemma}

\begin{lemma}[Guo and Li \cite{Guo_Li}]\label{Lemma4}
Let $n$ be a positive odd integer. Let $a_0$, $a_1$,$\cdots$, $a_{n-1}$ be a sequence of numbers satisfying
$a_k=-a_{(n-1)/2-k}$ for $0\leq k\leq (n-1)/2$ and $a_k=-a_{(3n-1)/2-k}$ for $(n+1)/2\leq k\leq n-1$.
Then
\begin{equation*}
\sum_{k=0}^{n-1}\sum_{j=0}^{k}a_ja_{k-j}=0.
\end{equation*}
\end{lemma}

\begin{proof}[Proof of Theorem \ref{Thm3}]
 Let
\begin{equation*}
c_q(k,a)=(-1)^k\frac{(1+q^{4k+1})(aq^{2},q^{2}/a;q^{4})_k}
{(1+q)(aq^{4},q^{4}/a;q^{4})_k}q^{2k^2+k}.
\end{equation*}
In terms of  congruence \eqref{EE7}, we have, for $0\leq k\leq (n-1)/2$,
\begin{equation*}
c_q((n-1)/2-k,a)\equiv -c_q(k,a)\pmod{\Phi_n(-q)}.
\end{equation*}
Similarly, for $(n+1)/2\leq k\leq n-1$,
\begin{equation*}
c_q((3n-1)/2-k,a)\equiv -c_q(k,a)\pmod{\Phi_n(-q)}.
\end{equation*}
Then, from Lemma \ref{Lemma4}, we can obtain
\begin{equation}\label{EE3}
\sum_{k=0}^{n-1}\sum_{j=0}^{k}c_q(j,a)c_q(k-j,a)\equiv0\pmod{\Phi_n(-q)}.
\end{equation}
On the other hand, setting $d=2$, $r=1$ in the congruence \eqref{eq10}, we derive,
modulo ${(1-aq^{2n})(a-q^{2n})}$,
\begin{flalign}
\sum_{k=0}^{n-1}\sum_{j=0}^{k}c_q(j,a)c_q(k-j,a)
\equiv
[n]^2_{q^2}q^{(n-1)^2}
\bigg(\sum_{k=0}^{\frac{n-1}{2}}\frac{(aq^{2},q^{2}/a;q^{4})_k}
{[4k+1](q^{4},q^{2};q^{4})_k}q^{2k}\bigg)^2
.\label{EEE9}
\end{flalign}
Because $n$ is odd, the reduced form of the right-hand side of congruence \eqref{EEE9} has the factor $\Phi_n(-q)$.
Then, combining congruence \eqref{EE3}, we conclude that \eqref{EEE9} is true modulo $\Phi_n(-q)(1-aq^{2n})(a-q^{2n})$.
Putting $a\rightarrow1$, we complete the proof of Theorem \ref{Thm3}.
\end{proof}

\begin{proof}[Proof of Theorem \ref{Thm4}]
Let
\begin{align*}
c_q(k,a)=(-1)^k\frac{(1+q^{4k+1})(q^2,aq^{2},q^{2}/a;q^{4})_k}
{(1+q)(q^4,aq^{4},q^{4}/a;q^{4})_k}q^{2k^2+2k}.
\end{align*}
When $(n+1)/2\leq k\leq n-1$, we can easily deduce
\begin{align*}
c_q((3n-1)/2-k,a)\equiv -c_q(k,a)\pmod{\Phi_n(-q)^2},
\end{align*}
where we have used the fact $(q^2,q^4)_k\equiv0 \pmod{\Phi_n(q^2)}$ for $(n+1)/2\leq k\leq n-1$.
Then, we conclude
\begin{align}\label{EE4}
\sum_{k=0}^{n-1}\sum_{j=0}^{k}c_q(j,a)c_q(k-j,a)
\equiv\bigg(\sum_{k=0}^{\frac{n-1}{2}}c_q(k,a)\bigg)^2
\equiv0
\pmod{\Phi_n(-q)^2},
\end{align}
where the last congruence is true since $\sum_{k=0}^{\frac{n-1}{2}}c_q(k,a)$ is a special case of
the left-hand side of \eqref{E4} which is congruent to $0$ modulo $\Phi_n(-q)$.

Besides, setting $d=2$ and $r=1$ in congruence \eqref{E6} gives,
modulo $(1-aq^{2n})(a-q^{2n})$,
\begin{flalign}
\sum_{k=0}^{n-1}\sum_{j=0}^{k}c_q(j,a)c_q(k-j,a)
\equiv
[n]^2_{q^2}q^{(n-1)^2}
\bigg(\sum_{k=0}^{\frac{n-1}{2}}\frac{(aq^{2},q^{2}/a,q^2;q^{4})_k}
{(q^{4},q^{3},q^{5};q^{4})_k}q^{2k}\bigg)^2.\label{eq1}
\end{flalign}
Since $n$ is odd, the reduced form of the right-hand side of congruence \eqref{eq1} also has the factor $\Phi_n(-q)^2$,
 which means that \eqref{eq1} is true modulo $\Phi_n(-q)^2$.  Then, we derive that \eqref{eq1} holds modulo $\Phi_n(-q)^2(1-aq^{2n})(a-q^{2n})$.
Finally, we finish the proof of Theorem \ref{Thm4} with $a\rightarrow1$.
\end{proof}

\section{Some $q$-supercongruences on triple sums}
Our main aim in this section is  to show how to prove
$q$-supercongruences on triple sums.
We first propose two $q$-supercongruences on triple  sums related to \eqref{EEE8} and \eqref{EEE7}.
\begin{theorem}\label{Thm5}
Let $n$ be a positive odd integer and $d\geq 3$  an integer with $n\equiv1\pmod d$.
Let
\begin{equation*}
c_q(k)=(-1)^k\frac{(1+q^{2dk+1})(q^2;q^{2d})^3_k}{(1+q)(q^{2d};q^{2d})^3_k}q^{dk^2+2k(d-1)},
\end{equation*}
then, for any nonnegative integer $i$, $j$, $s$, there holds, modulo $\Phi_n(-q)^3\Phi_n(q)^2$,
\begin{flalign*}
\sum_{\substack{i+j+s\leq n-1}}c_q(i)c_q(j)c_q(s)
\equiv
(-1)^{\frac{3(n-1)}{d}}[n]^3_{q^2}q^{\frac{3(n-1)(n+1-d)}{d}}\bigg(\sum_{k=0}^{\frac{n-1}{d}}
\frac{(q^2;q^{2d})^2_k(q^d;q^{2d})_k q^{2(d-1)k}}{(q^{d+1},q^{2d},q^{2d+1};q^{2d})_k}\bigg)^3.
\end{flalign*}
\end{theorem}
The proof of Theorem \ref{Thm5} will highly rely on the following lemma,
which is a special case of the result recently proposed by El Bachraoui \cite{Bachraoui2}.
Different from El Bachraoui's method, here,  a new proof of this lemma  shall be shown.
\begin{lemma}\label{Lemma6}
Let $n$ be a positive integer and $d\geq 3$, $r$ integers with $n\equiv r\pmod d$, $n-\lfloor d(n-1)/3\rfloor\leq r\leq n$.
Let $\{c(k)\}^\infty_{k=0}$ be a sequence of complex numbers.
If $c(k)=0$ for $(n-r)/d<k<n$, then, for any nonnegative integer $i$, $j$, $s$, there holds
\begin{equation}\label{EEEE3}
\sum_{\substack{i+j+s\leq n-1 }}c(i)c(j)c(s)=\bigg(\sum_{i=0}^{\frac{n-r}{d}}c(i)\bigg)^3.
\end{equation}
Furthermore, if $c(ln+k)/c(ln)=c(k)$ for all nonnegative integers $k$ and $l$ such that $0\leq k\leq n-1$, then
\begin{flalign}
&\sum_{\substack{i+j+s=ln+k}}c(i)c(j)c(s)\nonumber\\
&=\sum_{a=0}^{l}c(an)\sum_{t=0}^{l-a}c(tn)c((l-a-t)n)\sum_{i=0}^{k}c(i)\sum_{j=0}^{k-i}c(j)c(k-i-j)
.\label{EEEE4}
\end{flalign}
\end{lemma}
\begin{proof}
\eqref{EEEE3} is obvious from the assumption.
Next, we shall present the proof of  equality \eqref{EEEE4} via \eqref{eq5}.
\begin{flalign*}
&\sum_{\substack{i+j+s=ln+k}}c(i)c(j)c(s)\\
&=\sum_{i=0}^{k}c(i)\sum_{j=0}^{ln+k-i}c(j)c(ln+k-i-j)+\sum_{a=0}^{l-1}\sum_{i=an+k+1}^{(a+1)n+k}c(i)\sum_{j=0}^{ln+k-i}c(j)c(ln+k-i-j)\\
&=\sum_{t=0}^{l}c(tn)c((l-t)n)\sum_{i=0}^{k}c(i)\sum_{j=0}^{k-i}c(j)c(k-i-j)\\
&\quad+\sum_{a=0}^{l-1}\sum_{t=0}^{l-a-1}c(tn)c((l-a-t-1)n)
\sum_{i=an+k+1}^{(a+1)n+k}c(i)\sum_{j=0}^{(a+1)n+k-i}c(j)c((a+1)n+k-i-j)\\
&=\sum_{t=0}^{l}c(tn)c((l-t)n)\sum_{r_i=0}^{k}c(r_i)\sum_{j=0}^{k-r_i}c(j)c(k-r_i-j)\\
&\quad+\sum_{a=0}^{l-1}c(an)\sum_{t=0}^{l-a-1}c(tn)c((l-a-t-1)n)\sum_{r_i=k+1}^{n-1}c(r_i)\sum_{j=0}^{n+k-r_i}c(j)c(n+k-r_i-j)\\
&\quad+\sum_{a=0}^{l-1}c((a+1)n)\sum_{t=0}^{l-a-1}c(tn)c((l-a-t-1)n)\sum_{r_i=0}^{k}c(r_i)\sum_{j=0}^{k-r_i}c(j)c(k-r_i-j),
\end{flalign*}
where the second equality is obtained from \eqref{eq5}, and $r_i$ is the remainder of the division of $i$ by $n$.

Then, we claim
\begin{equation}\label{EEEEEE2}
\sum_{a=0}^{l-1}c(an)\sum_{t=0}^{l-a-1}c(tn)c((l-a-t-1)n)\sum_{r_i=k+1}^{n-1}c(r_i)\sum_{j=0}^{n+k-r_i}c(j)c(n+k-r_i-j)=0.
\end{equation}
In fact, when $r_i>k\geq (n-r)/d$ or $r_i\geq (n-r)/d>k$, we have
$$c(r_i)=0.$$
When $(n-r)/d>r_i>k$, there holds $$c(j)c(n+k-r_i-j)=0.$$
Thus, \eqref{EEEEEE2} is true and we get
\begin{flalign*}
&\sum_{\substack{i+j+s=ln+k}}c(i)c(j)c(s)\\
&=\sum_{t=0}^{l}c(tn)c((l-t)n)\sum_{r_i=0}^{k}c(r_i)\sum_{j=0}^{k-r_i}c(j)c(k-r_i-j)\\
&\quad+\sum_{a=0}^{l-1}c((a+1)n)\sum_{t=0}^{l-a-1}c(tn)c((l-a-t-1)n)\sum_{r_i=0}^{k}c(r_i)\sum_{j=0}^{k-r_i}c(j)c(k-r_i-j)\\
&=\sum_{a=0}^{l}c(an)\sum_{t=0}^{l-a}c(tn)c((l-a-t)n)\sum_{r_i=0}^{k}c(r_i)\sum_{j=0}^{k-r_i}c(j)c(k-r_i-j)
\end{flalign*}
as desired.
\end{proof}
\begin{proof}[Proof of Theorem \ref{Thm5}]
Let
\begin{equation*}
c_q(k,a)=(-1)^k\frac{(1+q^{2dk+1})(aq^{2},q^{2}/a;q^{2d})_k(q^{2};q^{2d})_k}
{(1+q)(aq^{2d},q^{2d}/a;q^{2d})_k(q^{2d};q^{2d})_k}q^{dk^2+2k(d-1)}.
\end{equation*}
Since $n\equiv 1\pmod{d}$, we have $(q^{2};q^{2d})_k\equiv0\pmod{\Phi_n(q^2)}$ for $(n-1)/d<k\leq n-1$.
From Lemma \ref{Lemma6}, we deduce
\begin{equation}\label{EEEE1}
\sum_{\substack{i+j+s\leq n-1 }}c_q(i,a)c_q(j,a)c_q(s,a)
\equiv
\bigg(\sum_{i=0}^{\frac{n-1}{d}}c_q(i,a)\bigg)^3
\equiv
0\pmod{\Phi_n(-q)}.
\end{equation}
Similarly, by making the substitution $a=q^{2n}$ or $a=q^{-2n}$, $r=1$ in \eqref{E4},
we obtain
\begin{align*}
&\sum_{k=0}^{\frac{n-1}{d}}(-1)^k\frac{(1+q^{2dk+1})(q^{2-2n},q^{2+2n},q^2;q^{2d})_k}{(1+q)(q^{2d-2n},q^{2d+2n},q^{2d};q^{2d})}q^{dk^2+2k(d-1)}\\
&=(-1)^{\frac{n-1}{d}}[n]_{q^2}q^{(n-1)(n+1-d)}{d}\sum_{k=0}^{\frac{n-1}{d}}\frac{(q^{2-2n},q^{2+2n},q^d;q^{2d})_k}{(q^{2d},q^{d+1},q^{2d+1};q^{2d})_k}q^{2k(d-1)}.
\end{align*}
Note  that $c_q(k,q^{2n})=c_q(k,q^{-2n})=0$ for $k$ in the range $(n-1)/d<k\leq n-1$.
Then, the use of Lemma \ref{Lemma6} gives
\begin{flalign*}
&\sum_{\substack{i+j+s\leq n-1 }}c_q(i,q^{2n})c_q(j,q^{2n})c_q(s,q^{2n})\\
&=
(-1)^{\frac{3(n-1)}{d}}[n]^3_{q^2}q^{\frac{3(n-1)(n+1-d)}{d}}
\bigg(\sum_{k=0}^{\frac{n-1}{d}}\frac{(q^{2+2n},q^{2-2n},q^d;q^{2d})_k q^{2k(d-1)}}
{(q^{2d},q^{d+1},q^{2d+1};q^{2d})_k}\bigg)^3.
\end{flalign*}
Namely,
 we  get the following $q$-congruence for the variable $a$: modulo $(1-aq^{2n})(a-q^{2n})$,
\begin{flalign}
&\sum_{\substack{i+j+s\leq n-1 }}c_q(i,a)c_q(j,a)c_q(s,a)\nonumber\\
&\equiv
(-1)^{\frac{3(n-1)}{d}}[n]^3_{q^2}q^{\frac{3(n-1)(n+1-d)}{d}}
\bigg(\sum_{k=0}^{\frac{n-1}{d}}\frac{(aq^{2},q^{2}/a,q^d;q^{2d})_k q^{2k(d-1)}}
{(q^{2d},q^{d+1},q^{2d+1};q^{2d})_k}\bigg)^3.\label{EEEE2}
\end{flalign}
It is easy to see that the right-hand side of \eqref{EEEE2} is also congruent to $0$ modulo $\Phi_n(-q)$.
As a result, the congruence \eqref{EEEE2} holds  modulo $\Phi_n(-q)$.
Since $1-aq^{2n}$, $a-q^{2n}$ and $\Phi_n(-q)$ are pairwise relatively prime polynomials,
we conclude \eqref{EEEE2} holds modulo $\Phi_n(-q)(1-aq^{2n})(a-q^{2n})$.
The proof then follows from the congruence \eqref{EEEE2} with $a\rightarrow 1$.
\end{proof}

Likewise, we can consider another result. Since the proof is exactly the same as that of Theorem \ref{Thm5},
we omit the details here.
\begin{theorem}
Let $n$ be a positive odd integer and $d\geq 3$ an integer with $n\equiv1\pmod d$.
Let
\begin{equation*}
c_q(k)=(-1)^k\frac{(1+q^{2dk+1})(q^2;q^{2d})^2_k}{(1+q)(q^{2d};q^{2d})^2_k}q^{dk^2+(d-1)k},
\end{equation*}
then, for any nonnegative integer  $i$, $j$, $s$, there holds, modulo $\Phi_n(-q)^2\Phi_n(q)^2$,
\begin{flalign*}
\sum_{\substack{i+j+s\leq n-1}}c_q(i)c_q(j)c_q(s)
\equiv
(-1)^{\frac{3(n-1)}{d}}[n]^3_{q^2}q^{\frac{3(n-1)(n+1-d)}{d}}\bigg(\sum_{k=0}^{\frac{n-1}{d}}
\frac{(q^2,q;q^{2d})_k q^{2(d-1)k}}{(q^{2d},q^{2d+1};q^{2d})_k}\bigg)^3.
\end{flalign*}
\end{theorem}

Besides, with the help of the Lemma \ref{Lemma6} and the Chinese remainder theorem, we can confirm the following result.
\begin{theorem}\label{THM1}
Let $n>1$, $d\geq3$ be integers with $n\equiv1\pmod d$. Let
\begin{equation*}
c_q(k)=[2dk+1]\frac{(q;q^d)^4_k}{(q^d;q^d)^4_k}q^{(d-2)k}.
\end{equation*}
Then, for any nonnegative integer $i$, $j$, $s$, there holds
\begin{align}\label{AE1}
\sum_{\substack{i+j+s\leq n-1}}c_q(i)c_q(j)c_q(s)
\equiv
[n]^3q^{\frac{3(1-n)}{d}}\frac{(q^2;q^d)^3_{\frac{n-1}{d}}}{(q^d;q^d)^3_{\frac{n-1}{d}}}\pmod{[n]\Phi_n(q)^3}.
\end{align}
\end{theorem}
In order to complete the proof of  Theorem \ref{THM1},
we first need to prove  following two lemmas basing on the results provided in \cite{Liu_Wang} and \cite{Wang_Xu1}.
\begin{lemma}\label{Lem5}
Let $n>1$, $d\geq3$ be integers with $n\equiv1\pmod d$. Let
\begin{equation*}
z_q(k)=[2dk+1]\frac{(q,aq,q/a,q/b;q^d)_k}{(q^d,aq^d,q^d/a,bq^d;q^d)_k}b^kq^{(d-2)k}.
\end{equation*}
Then, for nonnegative integers  $i$, $j$, $s$, there holds,
\begin{flalign*}
\sum_{\substack{i+j+s\leq n-1}}z_q(i)z_q(j)z_q(s)
\equiv
[n]^3\left(\frac{b}{q}\right)^{\frac{3(n-1)}{d}}\frac{(q^2/b;q^d)^3_{\frac{n-1}{d}}}{(bq^d;q^d)^3_{\frac{n-1}{d}}}
\pmod{[n](1-aq^n)(a-q^n)}
.
\end{flalign*}
\end{lemma}
\begin{proof}
Substituting $r=1$ and $ce=q^{d-1}$ into \cite[Lemma 2]{Liu_Wang} gives,
\begin{flalign}
\sum_{k=0}^{M}z_q(k)
\equiv[n]\left(\frac{b}{q}\right)^{\frac{n-1}{d}}
\frac{(q^{2}/b;q^d)_{\frac{n-1}{d}}}{(bq^d;q^d)_{\frac{n-1}{d}}}\pmod{\Phi_n(q)(1-aq^n)(a-q^n)}
,\label{Equation3}
\end{flalign}
where $M=(n-1)/d$ or $n-1$.
It is easy to see that $(q;q^d)_k\equiv0\pmod{\Phi_n(q)}$ for $(n-1)/d<k<n$. Then, there holds
$z_q(k)\equiv0 \pmod{\Phi_n(q)}$ for $(n-1)/d<k<n$. Thus, employing  Lemma \ref{Lemma6} produces
\begin{equation}\label{Equation4}
\sum_{\substack{i+j+s\leq n-1}}z_q(i)z_q(j)z_q(s)
\equiv
[n]^3\left(\frac{b}{q}\right)^{\frac{3(n-1)}{d}}\frac{(q^2/b;q^d)^3_{\frac{n-1}{d}}}{(bq^d;q^d)^3_{\frac{n-1}{d}}} \pmod{\Phi_n(q)}.
\end{equation}
Let $\zeta\neq1$ be an $n$-th root of unity, not necessarily primitive.
Namely, $\zeta$ is a primitive $m$-th root of unity with $m\mid n$ and $m>1$.
From \eqref{Equation4}, we get
\begin{equation}\label{EEEE5}
\sum_{\substack{i+j+s\leq m-1}}z_{\zeta}(i)z_{\zeta}(j)z_{\zeta}(s)=0.
\end{equation}
It is not difficult to check $z_{\zeta}(lm+k)/z_{\zeta}(lm)=z_{\zeta}(k)$ for nonnegative integers $l$ and $k$ with $0\leq k\leq m-1$.
By Lemma \ref{Lemma6} again, we have
\begin{flalign}
&\sum_{\substack{i+j+s\leq n-1}}z_{\zeta}(i)z_{\zeta}(j)z_{\zeta}(s)\nonumber\\
&=\sum_{r=0}^{n-1}\sum_{i=0}^{r}z_{\zeta}(i)\sum_{j=0}^{r-i}z_{\zeta}(j)z_{\zeta}(r-i-j)\nonumber\\
&=\sum_{l=0}^{\frac{n}{m}-1}\sum_{k=0}^{m-1}\sum_{i=0}^{lm+k}z_{\zeta}(i)\sum_{j=0}^{lm+k-i}z_{\zeta}(j)z_{\zeta}(lm+k-i-j)\nonumber\\
&=\sum_{l=0}^{\frac{n}{m}-1}\sum_{k=0}^{m-1}
\sum_{a=0}^{l}z_{\zeta}(am)\sum_{t=0}^{l-a}z_{\zeta}(tm)z_{\zeta}((l-a-t)m)
\sum_{i=0}^{k}z_{\zeta}(i)\sum_{j=0}^{k-i}z_{\zeta}(j)z_{\zeta}(k-i-j)
\nonumber\\
&=\sum_{l=0}^{\frac{n}{m}-1}\sum_{a=0}^{l}z_{\zeta}(am)\sum_{t=0}^{l-a}z_{\zeta}(tm)z_{\zeta}((l-a-t)m)
\sum_{k=0}^{m-1}\sum_{i=0}^{k}z_{\zeta}(i)
\sum_{j=0}^{k-i}z_{\zeta}(j)z_{\zeta}(k-i-j)\nonumber\\
&=0.\nonumber
\end{flalign}
Since the above equality is true for any $n$-th root of unity $\zeta\neq1$, we conclude that
\begin{equation}\label{EEEE6}
\sum_{\substack{i+j+s\leq n-1}}z_{q}(i)z_{q}(j)z_{q}(s)\equiv0\pmod{[n]}.
\end{equation}
On the other hand, letting $a=q^n$ or $q^{-n}$ in \eqref{Equation3},  we have
\begin{equation*}
\sum_{k=0}^{M}\tilde{z}_q(k)
=[n]\left(\frac{b}{q}\right)^{\frac{n-1}{d}}
\frac{(q^{2}/b;q^d)_{\frac{n-1}{d}}}{(bq^d;q^d)_{\frac{n-1}{d}}},
\end{equation*}
where $\tilde{z}_q(k)$ is defined by
\begin{equation*}
\tilde{z}_q(k)=[2dk+1]\frac{(q,q^{1+n},q^{1-n},q/b;q^d)_k}{(q^d,q^{d+n},q^{d-n},bq^d;q^d)_k}b^kq^{(d-2)k}.
\end{equation*}
Note the fact that $\tilde{z}_q(k)=0$ for $(n-1)/d<k<n$. Using Lemma \ref{Lemma6} again, we get
\begin{equation*}
\sum_{\substack{i+j+s\leq n-1}}\tilde{z}_q(i)\tilde{z}_q(j)\tilde{z}_q(s)
=
[n]^3\left(\frac{b}{q}\right)^{\frac{3(n-1)}{d}}\frac{(q^2/b;q^d)^3_{\frac{n-1}{d}}}{(bq^d;q^d)^3_{\frac{n-1}{d}}}.
\end{equation*}
Namely, we have the following $q$-congruence for the variable $a$: modulo $(1-aq^n)(a-q^n)$,
\begin{equation}\label{Equation6}
\sum_{\substack{i+j+s\leq n-1}}z_q(i)z_q(j)z_q(s)
\equiv
[n]^3\left(\frac{b}{q}\right)^{\frac{3(n-1)}{d}}\frac{(q^2/b;q^d)^3_{\frac{n-1}{d}}}{(bq^d;q^d)^3_{\frac{n-1}{d}}}.
\end{equation}
Then, by \eqref{EEEE6},
we conclude \eqref{Equation6} holds modulo $[n](1-aq^n)(a-q^n)$
 as desired.
\end{proof}

\begin{lemma}\label{Lem6}
Let $n>1$, $d\geq3$ be integers with $n\equiv1\pmod d$. Let
\begin{equation*}
z_q(k)=[2dk+1]\frac{(q,aq,q/a,q/b;q^d)_k}{(q^d,aq^d,q^d/a,bq^d;q^d)_k}b^kq^{(d-2)k}.
\end{equation*}
For nonnegative integers  $i$, $j$, $s$, there holds,
\begin{flalign*}
\sum_{\substack{i+j+s\leq n-1}}z_q(i)z_q(j)z_q(s)
\equiv
[n]^3\frac{(q,q^{d-1};q^d)^3_{\frac{n-1}{d}}}{(aq^d,q^d/a;q^d)^3_{\frac{n-1}{d}}} \pmod{b-q^n}.
\end{flalign*}
\end{lemma}
\begin{proof}
Setting $r=1$ and $ce=q^{d-1}$  in \cite[Lemma 2.4]{Wang_Xu1} gives
\begin{equation}\label{Equation7}
\sum_{k=0}^{M}z_q(k)
\equiv
[n]
\frac{(q,q^{d-1};q^d)_{\frac{n-1}{d}}}{(aq^d,q^d/a;q^d)_{\frac{n-1}{d}}} \pmod{(b-q^n)},
\end{equation}
where $M=(n-1)/d$ or $n-1$.
Letting $b=q^n$ in \eqref{Equation7}, we get
\begin{equation*}
\sum_{k=0}^{M}\hat{z}_q(k)
=
[n]
\frac{(q,q^{d-1};q^d)_{\frac{n-1}{d}}}{(aq^d,q^d/a;q^d)_{\frac{n-1}{d}}},
\end{equation*}
where $\hat{z}_q(k)$ denotes
\begin{equation*}
\hat{z}_q(k)=[2dk+1]\frac{(q,aq,q/a,q^{1-n};q^d)_k}{(q^d,aq^d,q^d/a,q^{d+n};q^d)_k}q^{(n+d-2)k}.
\end{equation*}
Note that $\hat{z}_q(k)=0$ for $(n-1)/d<k<n$. From Lemma \ref{Lemma6}, we obtain
\begin{equation*}
\sum_{\substack{i+j+s\leq n-1}}\hat{z}_q(i)\hat{z}_q(j)\hat{z}_q(s)
=
[n]^3
\frac{(q,q^{d-1};q^d)^3_{\frac{n-1}{d}}}{(aq^d,q^d/a;q^d)^3_{\frac{n-1}{d}}},
\end{equation*}
which means
\begin{equation*}
\sum_{\substack{i+j+s\leq n-1}}z_q(i)z_q(j)z_q(s)
\equiv
[n]^3
\frac{(q,q^{d-1};q^d)^3_{\frac{n-1}{d}}}{(aq^d,q^d/a;q^d)^3_{\frac{n-1}{d}}}\pmod{(b-q^n)}
\end{equation*}
holds.
\end{proof}

\begin{proof}[Proof of Theorem \ref{THM1}]
It is easy to see that $[n](1-aq^n)(a-q^n)$ and $b-q^n$ are relatively prime polynomials. Noting the relations
\begin{align*}
\frac{(b-q^n)(ab-1-a^2+aq^n)}{(a-b)(1-ab)}&\equiv 1 \pmod{(1-aq^n)(a-q^n)},\\
\frac{(1-aq^n)(a-q^n)}{(a-b)(1-ab)}&\equiv 1 \pmod{b-q^n},
\end{align*}
and employing the Chinese remainder theorem for coprime polynomials, we derive the following result
from Lemma \ref{Lem5} and Lemma \ref{Lem6}: modulo $[n](1-aq^n)(a-q^n)(b-q^n)$,
\begin{flalign}
&\sum_{\substack{i+j+s\leq n-1}}z_q(i)z_q(j)z_q(s)\nonumber\\
&\equiv
[n]^3\bigg\{
\frac{(b-q^n)(ab-1-a^2+aq^n)}{(a-b)(1-ab)}\bigg(\frac{b}{q}\bigg)^{\frac{3(n-1)}{d}}
\frac{(q^2/b;q^d)^3_{\frac{n-1}{d}}}{(bq^d;q^d)^3_{\frac{n-1}{d}}}\nonumber\\
&+\frac{(1-aq^n)(a-q^n)}{(a-b)(1-ab)}\frac{(q,q^{d-1};q^d)^3_{\frac{n-1}{d}}}{(aq^d,q^d/a;q^d)^3_{\frac{n-1}{d}}}\bigg\}
.\label{Eequation6}
\end{flalign}
Taking $b=1$ in \eqref{Eequation6}, we deduce that, modulo $[n]\Phi_n(q)(1-aq^n)(a-q^n)$,
\begin{flalign}
\sum_{\substack{i+j+s\leq n-1}}z_q(i)z_q(j)z_q(s)
\equiv
[n]^3
q^{\frac{3(1-n)}{d}}
\frac{(q^2;q^d)^3_{\frac{n-1}{d}}}{(q^d;q^d)^3_{\frac{n-1}{d}}}
.\label{Eequation7}
\end{flalign}
Here we have employed the relation:
\begin{equation*}
(1-q^n)(1+a^2-a-aq^n)=(1-a)^2+(1-aq^n)(a-q^n).
\end{equation*}
Finally, letting $a=1$ in \eqref{Eequation7} and noting $1-q^n$ has the factor $\Phi_n(q)$,
we can see that the  congruence \eqref{AE1} is true modulo $\Phi_n(q)^4$. On the other hand, our proof of
\eqref{EEEE6} is still valid for $a=1$, which means \eqref{AE1} is also correct modulo $[n]$.
Since the least common multiple  of $\Phi_n(q)^4$ and $[n]$ is $[n]\Phi_n(q)^3$, we finish the proof of Theorem \ref{THM1}.
\end{proof}

In \cite{He}, He  obtained the following Ramanujan-type supercongruence: for $p\geq5$,
\begin{equation*}
\sum_{k=0}^{(p-1)/3}(6k+1)\frac{(1/3)^4_k}{k!^4}\equiv -p\Gamma_p(1/3)^3\pmod {p^4}, \quad\mathrm{if}\ p\equiv1\pmod3.
\end{equation*}
We here give the corresponding $q$-supercongruence on triple sums by making the substitution $d=3$ in Theorem \ref{THM1}.
 \begin{corollary}
Let $n>1$ be an integer with $n\equiv1\pmod3$ and $c_q(k)$ be the $k$-th term of
\begin{equation*}
\sum_{k=0}^{\infty}[6k+1]\frac{(q;q^3)^4_k}{(q^3;q^3)^4_k}q^{k}.
\end{equation*}
For nonnegative integers  $i$, $j$, $s$, there holds, modulo $[n]\Phi_n(q)^3$,
\begin{flalign*}
\sum_{\substack{i+j+s\leq n-1}}c_q(i)c_q(j)c_q(s)
\equiv
[n]^3q^{1-n}\frac{(q^2;q^3)^3_{\frac{n-1}{3}}}{(q^3;q^3)^3_{\frac{n-1}{3}}}
.
\end{flalign*}
\end{corollary}

Similarly, setting $d=4$ in Theorem \ref{THM1} produces a new $q$-supercongruence related to van Hamme's (G.2).
\begin{corollary}
Let $n>1$ be an integer with $n\equiv1\pmod4$ and $c_q(k)$ be the $k$-th term of
\begin{equation*}
\sum_{k=0}^{\infty}[8k+1]\frac{(q;q^4)^4_k}{(q^4;q^4)^4_k}q^{2k}.
\end{equation*}
For nonnegative integers  $i$, $j$, $s$, there holds, modulo $[n]\Phi_n(q)^3$,
\begin{flalign}
\sum_{\substack{i+j+s\leq n-1}}c_q(i)c_q(j)c_q(s)
\equiv
[n]^3q^{\frac{3(1-n)}{4}}\frac{(q^2;q^4)^3_{\frac{n-1}{4}}}{(q^4;q^4)^3_{\frac{n-1}{4}}}
.\label{Equation1}
\end{flalign}
\end{corollary}



\begin{thebibliography}{99}
\bibitem{Andrews}G.E. Andrews,
Ramanujan's ``lost'' notebook. IV. Stacks and alternating parity in partitions,
{\em Adv. in Math.} \textbf{53} (1984), 55--74.

\bibitem{Bachraoui}M. El Bachraoui,
On supercongruences for truncated sums of squares of basic hypergeometric series,
{\em Ramanujan J.} \textbf{54} (2021), 415--426.

\bibitem{Bachraoui2}M. El Bachraoui,
N-tuple sum analogues for Ramanujan-type congruences,
arXiv: 2112. 00308.

\bibitem{Guo5}V.J.W. Guo,
A $q$-analogue of a Ramanujan-type supercongruence involving central binomial coefficients,
{\em J. Math. Anal. Appl.} \textbf{458} (2018), 590--600.


\bibitem{Guo2}V.J.W. Guo, A $q$-analogue of the (L.2) supercongruence of Van Hamme,
{\em J. Math. Anal. Appl.} \textbf{466} (2018), 749--761.


\bibitem{Guo7}V.J.W. Guo,
$q$-Analogues of the (E.2) and (F.2) supercongruences of van Hamme,
{\em Ramanujan J.} \textbf{49} (2019), 531--544.

\bibitem{Guo6}V.J.W. Guo,
Some $q$-supercongruences from the Bailey transformation,
 {\em Period. Math. Hungar.} to appear.

\bibitem{Guo_Li}V.J.W. Guo and L. Li,
$q$-Supercongruences from squares of basic hypergeometric series,
preprint, 2021.

\bibitem{Guo_Schlosser1}V.J.W. Guo and M.J. Schlosser, A new family of $q$-supercongruences modulo the fourth
power of a cyclotomic polynomial,
 {\em Results Math.} \textbf{75} (2020), Art. 155.


\bibitem{Guo_Zudilin}V.J.W. Guo and W. Zudilin,
A $q$-microscope for supercongruences,
 \emph{Adv. Math.}
 \textbf{346} (2019), 329--358.

\bibitem{He}B. He, Supercongruences on truncated hypergeometric Series, \emph{Results Math.} \textbf{72} (2017), 303--317.

\bibitem{LiLong}L. Li,
Some $q$-supercongruences for truncated forms of squares of basic hypergeometric series,
 {\em J. Difference Equ. Appl.} \textbf{27} (2021), 16--25.

\bibitem{Liu_Wang}Y. Liu and X. Wang, $q$-Analogues of two Ramanujan-type supercongrucences,
\emph{J. Math. Anal. Appl.} \textbf{502} (2021), Art. 125238.

\bibitem{Ramanujan}S. Ramanujan, Modular equations and approximations to $\pi$.
{\em Quart. J. Math.}
\textbf{45} (1914), 350--372.

\bibitem{Sun}Z.-W. Sun,
Binomial coefficients, Catalan numbers and Lucas quotients,
{\em Sci. China Math.} \textbf{53} (2010), 2473--2488.


\bibitem{Van}L. Van Hamme, Some conjectures concerning partial sums of generalized hypergeometric series,
in: \emph{$p$-Adic Functional Analysis} (Nijmegen, 1996),
Lecture Notes in Pure and Appl. Math. \textbf{192}, Dekker, New York (1997), 223--236.



\bibitem{Wang_Xu}X. Wang and C. Xu,
New $q$-supercongruences from the Bailey transformation,
preprint, 2021.

\bibitem{Wang_Xu1}X. Wang and C. Xu,
New $q$-supercongruences related to Ramanujan-type supercongruences,
preprint, 2021.


\bibitem{Wang_Yu}X. Wang and M. Yu, Some new $q$-congruences on double sums.
\emph{Rev. R. Acad. Cienc. Exactas F{\'i}s. Nat., Ser. A Mat. RACSAM} \textbf{115} (2021), Art. 9.

\bibitem{Wang_Yue}X. Wang and M. Yue, Some $q$-supercongruences from Watson's $_8\phi_7$ transformation formula,
\emph{Results Math.} \textbf{75} (2020), Art. 71.

\bibitem{Xu_Wang}C. Xu and X. Wang,
Proofs of Guo and Schlosser's two conjectures,
{\em Period. Math. Hungar.} to appear.
\end{thebibliography}
\end{document}